\documentclass[11pt,a4paper]{amsart}
\usepackage{enumerate,amsmath,amssymb,amsthm}     
\usepackage[linkbordercolor={1 1 1}]{hyperref}
\usepackage{lscape}
\usepackage{enumitem}
\usepackage[pdftex]{graphicx}
\usepackage{array}

\theoremstyle{plain}      
 
\newtheorem{thm}{Theorem}[section]     
\newtheorem{theorem}[thm]{Theorem}

\newtheorem{lemma}[thm]{Lemma}     
     
\newtheorem{proposition}[thm]{Proposition}     
     
\theoremstyle{remark} 
\newtheorem{example}[thm]{Example} 
\newtheorem{remark}[thm]{Remark}

\theoremstyle{definition}      
     
\newtheorem{definition}[thm]{Definition}     

\DeclareMathAlphabet{\doba}{U}{msb}{m}{n}

\def\dd{\mathrm{d}}

\let\<\langle 
\let\>\rangle

\newcommand{\definedas}{\mathrel{\raise.095ex\hbox{\rm :}\mkern-5.2mu=}}

\title{Toric Vaisman Manifolds}
\author{Mihaela Pilca}

\address{Mihaela Pilca\\Fakult\"at f\"ur Mathematik\\
Universit\"at Regensburg\\Universit\"atsstr. 31 
D-93040 Regensburg, Germany
\emph{and} 
Institute of Mathematics ``Simion Stoilow" of the Romanian Academy, 
21, Calea Grivitei Str.
010702-Bucharest, Romania}
\email{mihaela.pilca@mathematik.uni-regensburg.de}

\subjclass[2010]{}

\begin{document}

\begin{abstract}
Vaisman manifolds are strongly related to K\"ahler and Sasaki geo\-metry. 
In this paper we introduce  
toric Vaisman structures and show that 
this relationship still holds in the toric context. It is known that the so-called minimal covering of a Vaisman manifold is the Riemannian cone over a Sasaki manifold. We show that if a complete Vaisman manifold is toric, then the associated Sasaki manifold is also toric. Conversely, a toric complete Sasaki manifold, whose K\"ahler cone is equipped with an appropriate compatible action, gives rise to a toric Vaisman manifold. In the special case of a strongly regular compact Vaisman manifold, we show that it is toric if and only if the corresponding K\"ahler quotient is toric.
\end{abstract}

\maketitle

\section{Introduction}

Toric geometry has been studied intensively, as mani\-folds with many symmetries often occur in physics and also represent a large source of examples as testing ground for conjectures. 
The classical case of compact symplectic toric manifolds has been completely classified by T. Delzant \cite{A5Delzant1988}, who showed that they are in one-to-one correspondence to the so-called Delzant polytopes, obtained as the image of the momentum map. Afterwards, si\-mi\-lar classification results have been given in many different geometrical settings, some of which we briefly mention here. For instance, classification results were obtained by Y.~Karshon and E.~Lerman \cite{A5KL2015} for non-compact symplectic toric manifolds and by E.~Lerman and S.~Tolman \cite{A5LT1997} for symplectic orbifolds. The case when one additionally considers compatible metrics invariant under the toric action is also well understood: compact toric K\"ahler manifolds have been investigated by V.~Guillemin \cite{A5Guillemin1994}, D.~Calderbank, L.~David and P.~Gauduchon \cite{A5CDG2003},  M.~Abreu \cite{A5Abreu1998},
and compact toric K\"ahler orbifolds in \cite{A5Abreu2001}. Other more special structures have been completely classified, such as orthotoric K\"ahler, by V.~Apostolov, D.~Calderbank and  P.~Gauduchon \cite{A5ACG2006} or toric hyperk\"ahler, by R. Bielawski and  A. Dancer \cite{A5BD2000}.  The odd-dimensional counterpart, namely the compact contact toric manifolds, are classified by E.~Lerman \cite{A5L2003}, whereas toric Sasaki manifolds were also studied by M.~Abreu  \cite{A5Abreu2010}, \cite{A5BG2008}. These were used to produce examples of compact Sasaki-Einstein manifolds, for instance by D.~Martelli, J.~Sparks and S.-T.~Yau \cite{A5MSY2006}, A.~Futaki, H.~Ono and G.~Wang~\cite{A5FOW2009}, C.~van~Coevering~\cite{A5vC2009}.

In the present paper, we consider toric geometry in the context of locally conformally K\"ahler manifolds. These are defined as complex manifolds admitting a compatible metric, which, on given charts, is conformal to a local K\"ahler metric. The differentials of the logarithms of the conformal factors glue up to a well-defined closed $1$-form, called the Lee form. We are mostly interested in the special class of so-called Vaisman manifolds, defined by the additional property of having parallel Lee form. By analogy to the other geometries, we introduce the notion of toric locally conformally K\"ahler manifold. More precisely, we require the existence of an effective torus action of dimension half the dimension of the manifold, which preserves the holomorphic structure and is twisted Hamiltonian. I.~Vaisman \cite{A5Vaisman1985} introduced twisted Hamiltonian actions and 
they have been used for instance by S.~Haller and T.~Rybicki \cite{A5HR2001} and
by R. Gini, L. Ornea and M. Parton \cite{A5GOP2005}, where reduction results for locally symplectic, respectively locally conformal K\"ahler manifolds are given, or more recently by A.~Otiman \cite{A5Otiman}.

Vaisman geometry is closely related to both Sasaki and K\"ahler geometry. In fact, a locally conformally K\"ahler manifold may be equivalently defined as a manifold whose universal covering is K\"ahler and on which the fundamental group acts by holomorphic homotheties. For Vaisman manifolds, the universal and the minimal covering are K\"ahler cones over Sasaki manifolds, as proven in \cite{A5Vaisman1979}, \cite{A5GOPP2006}. 
On the other hand, in the special case of strongly regular compact Vaisman manifolds, the quotient by the $2$-dimensional distribution spanned by the Lee and anti-Lee vector fields is a K\"ahler manifold, \emph{cf.} \cite{A5Vaisman1982}.

The purpose of this paper is to make a first step towards a possible classification of toric Vaisman, or more generally, toric locally conformally K\"ahler manifolds, by showing that the above mentioned connections between Vaisman and Sasaki, respectively K\"ahler manifolds are still true when imposing the toric condition. For the precise statements of these equivalences, we refer to Theorem~\ref{A5vaissas}, Theorem~\ref{A5sasvais} and Theorem~\ref{A5vaisreg}.

{\it Acknowledgement.} I would like to thank P.~Gauduchon and A.~Moroianu for useful discussions and L.~Ornea for introducing me to locally conformally K\"ahler geometry. 

\section{Preliminaries}
A {\it locally conformally K\"ahler manifold} (shortly lcK) is a conformal Hermitian manifold $(M^{2n}, [g], J)$ of complex dimension $n\geq 2$, such that for one (and hence for all) metric $g$ in the conformal class, the corresponding fundamental $2$-form $\omega:=g(\cdot,J\cdot)$ satisfies: $\mathrm{d} \omega=\theta\wedge\omega$, with $\theta$ a closed $1$-form, called the {\it Lee form} of the Hermitian structure $(g,J)$. 
Equivalently, there exists an atlas on $M$, such that the restriction of $g$ to any chart is conformal to a K\"ahler metric. In fact, the differential of the logarithmic of the conformal factors are, up to a constant, equal to the  Lee form. It turns out to be convenient to denote also by $(M,g,J,\theta)$ an lcK manifold, when fixing one metric $g$ in the conformal class. By $\nabla$ we denote the Levi-Civita connection of $g$.

We denote  by $\theta^\sharp$
the vector field dual to $\theta$ with respect to the metric $g$, the so-called {\it Lee vector field} of the lcK structure, and by $J\theta^\sharp$ the {\it anti-Lee vector field}.

\begin{remark}\label{A5nablaJ}
 On an lcK manifold $(M,g,J,\theta)$, the following formula for the covariant derivative of $J$ holds:
 \[2\nabla_XJ = X\wedge J\theta^\sharp +JX\wedge \theta^\sharp,\quad \forall X\in\mathfrak{X}(M),\]
 or, more explicitly, applied to any vector field $Y\in \mathfrak{X}(M)$:
 \[2(\nabla_XJ)(Y)=\theta(JY)X-\theta(Y)JX+g(JX,Y)\theta^\sharp+g(X,Y)J\theta^\sharp.\]
 In particular, it follows that 
 $\nabla_{\theta^\sharp}J=0$ and $\nabla_{J\theta^\sharp}J=0$.
\end{remark}

\begin{remark}\label{A5omth}
On an lcK manifold $(M^{2n},g,J,\theta)$, a vector field $X$ preserving the fundamental $2$-form $\omega$, also preserves the Lee form, \emph{i.e.} $\mathcal{L}_X\omega=0$ implies $\mathcal{L}_X\theta=0$, as follows.
As the differential and the Lie derivative with respect to a vector field commute to each other,  \emph{e.g.} by the Cartan formula, we obtain:
\[0=d(\mathcal{L}_X\omega)=\mathcal{L}_X(d\omega)=\mathcal{L}_X(\theta\wedge\omega)=\mathcal{L}_X\theta\wedge\omega+\theta\wedge\mathcal{L}_X \omega=\mathcal{L}_X\theta\wedge\omega.\]
Since the map form $\Omega^1(M)$ to $\Omega^3(M)$ given by wedging with $\omega$ is injective,  for complex dimension $n\geq 2$, it follows that $\mathcal{L}_X\theta=0$.
\end{remark}

We now recall the definition of Vaisman manifolds, which were first introduced and studied by I. Vaisman \cite{A5Vaisman1979}, \cite{A5Vaisman1982}:

\begin{definition}
A {\it Vaisman manifold} is an lcK manifold $(M,g,J,\theta)$ admitting a metric in the conformal class, such that its Lee form is parallel, \emph{i.e.}  $\nabla\theta =0$.  
\end{definition}

Note that on a compact lcK manifold, a metric with parallel Lee
form $\theta$, if it exists, is unique up to homothety in its conformal class and coincides
with the so-called {\it Gauduchon metric}, \emph{i.\,e.} the metric with co-closed Lee form: $\delta\theta=0$. In this paper, we scale any Vaisman metric $g$ such that the norm of its Lee vector field $\theta^\sharp$, which is constant since $\theta$ is parallel, equals $1$.

\begin{definition}
The automorphism group of a Vaisman manifold $(M,g,J,\theta)$ is denoted by a slight abuse of notation
$\mathrm{Aut}(M):=\mathrm{Aut}(M,g,J,\theta)$
and is defined as the group of conformal biholomorphisms:
\[\mathrm{Aut}(M)=\{ F\in \mathrm{Diff}(M)\,|\,  F^{*}J=J, [F^*g]=[g] \}.\]
\end{definition}

We emphasize here that we define the group of automorphisms like for lcK manifolds, namely we do not ask for the automorphisms of a Vaisman manifold to be isometries of the Vaisman metric, but only to preserve its conformal class. Hence, the Lie algebra of $\mathrm{Aut}(M)$ is:
\begin{equation}
\mathfrak{aut}(M)=\{X\in\mathfrak{X}(M)\,|\, \mathcal{L}_X J=0, \mathcal{L}_X g=f g, \text{ for some } f\in\mathcal{C}^\infty(M)\}.
\end{equation}
We denote by $\mathfrak{isom}(M)$  and $\mathfrak{hol}(M)$ the Lie algebras of Killing vector fields with respect to the Vaisman metric $g$, respectively of holomorphic vector fields.

The following lemma collects some known properties of Vaisman manifolds that are used in the sequel.

\begin{lemma}[]\label{A5thJth}
On a Vaisman manifold $(M,J,g,\theta)$, the Lee vector field $\theta^\sharp$ and the anti-Lee vector field $J\theta^\sharp$ are both holomorphic Killing vector fields:
\[\mathcal{L}_{\theta^\sharp} g=0,\quad \mathcal{L}_{\theta^\sharp} J=0, \quad \mathcal{L}_{J\theta^\sharp} g=0, \quad \mathcal{L}_{J\theta^\sharp} J=0,\]
so that $\theta^\sharp, J\theta^\sharp\in\mathfrak{isom}(M)\cap\mathfrak{hol}(M)$.
In particular, it follows that $\mathcal{L}_{\theta^\sharp} \omega=0$, $\mathcal{L}_{J\theta^\sharp} \omega=0$ and $\theta^\sharp$ commutes with $J\theta^\sharp$:
\[[\theta^\sharp, J\theta^\sharp]=J[\theta^\sharp, \theta^\sharp]=0.\]
The Lee vector field $\theta^\sharp$ and the anti-Lee vector field $J\theta^\sharp$  span a $1$-dimensional complex Lie subalgebra of $\mathfrak{isom}(M)\cap\mathfrak{hol}(M)$, which moreover lies in the center of $\mathfrak{isom}(M)\cap\mathfrak{hol}(M)$. Thus, these vector fields give rise to the so-called canonical foliation $\mathcal{F}$ by $1$-dimensional complex tori on a Vaisman manifold, which is a totally
geodesic Riemannian foliation. 
\end{lemma}
\begin{proof}

The following relation holds for any integrable complex structure $J$: $\mathcal{L}_{JX}J=J\circ\mathcal{L}_XJ$, for all  vector fields $X$. This shows that a vector field $X$ is holomorphic if and only if $JX$ is holomorphic. In particular, we obtain that $J\theta^\sharp$ is a holomorphic vector field, because $\theta^\sharp$ is holomorphic.
We now show that $J\theta^\sharp$ is a Killing vector field. By the formula in Remark~\ref{A5nablaJ}, we obtain for $Y=\theta^\sharp$ and $X\in\mathfrak{X}(M)$:
 \begin{equation}
  \begin{split}
   2(\nabla_XJ)(\theta^\sharp)&=-\theta(\theta^\sharp)JX+\theta(J\theta^\sharp)X+g(JX,\theta^\sharp)\theta^\sharp+g(X,\theta^\sharp)J\theta^\sharp\\
   &=-JX-g(X,J\theta^\sharp)\theta^\sharp+g(X,\theta^\sharp)J\theta^\sharp.
  \end{split}
 \end{equation}
Hence, it follows that $\nabla J\theta^\sharp=(\nabla J)(\theta^\sharp)$ (since $\theta^\sharp$ is parallel) is skew-symmetric:
 \begin{equation}
  \begin{split}
   2g((\nabla_X J)(\theta^\sharp),Y)&=-g(JX,Y)-g(X,J\theta^\sharp)g(Y,\theta^\sharp)+g(X,\theta^\sharp)g(Y,J\theta^\sharp)\\
   &=g(JY,X)+g(Y,J\theta^\sharp)g(X,\theta^\sharp)-g(Y,\theta^\sharp)g(X,J\theta^\sharp)\\
   &=-2g((\nabla_Y J)(\theta^\sharp),X).
  \end{split}
 \end{equation}
Thus, $J\theta^\sharp$ is a Killing vector field of the Vaisman metric $g$. In particular, it follows:
\[\{\theta^\sharp,J\theta^\sharp\}\subset\mathfrak{isom}(M)\cap\mathfrak{hol}(M).\]
Let $X$ be a holomorphic Killing vector field. Thus, $X$ also preserves the fundamental form $\omega$ and by Remark~\ref{A5omth} it holds: $\mathcal{L}_X\theta=0$. Furthermore, the following general equality
$(\mathcal{L}_X g)(\cdot,\theta^\sharp)=\mathcal{L}_X \theta - g(\cdot, \mathcal{L}_X \theta^\sharp)$
shows that \[[X,\theta^\sharp]=\mathcal{L}_X \theta^\sharp=0.\]
Since $X$ is holomorphic, it also follows that
\[[X,J\theta^\sharp]=J[X,\theta^\sharp]=0.\]
Thus, $\{\theta^\sharp,J\theta^\sharp\}$ is a subset of the center of $\mathfrak{isom}(M)\cap\mathfrak{hol}(M)$.
\end{proof}
 
Recall that the {\it cone} over a Riemannian manifold $(W,g_W)$ is defined as $\mathcal{C}(W):=\mathbb{R}\times W$ with the metric $g_{cone}:=4e^{-2t}(dt^2+ p^*g_W)$, where $t$ is the parameter on $\mathbb{R}$ and $p\colon\mathbb{R}\times W\to W$ is the projection on $W$. We denote the radial flow on the cone by $\phi$, \emph{i.e.} $\phi_s\colon \mathcal{C}(W)\to \mathcal{C}(W)$, $\phi_s(t,w):=(t+s, w)$, for each $s\in \mathbb{R}$. Note that this definition is equivalent, up to a constant factor\footnote{We choose to multiply the metric by the constant factor $4$, so that later on a Vaisman manifold obtained as a quotient of a cone, the Vaisman metric has the property that its Lee vector field is of length $1$ (see Section \ref{A5secregvais}).}, to the more common definition in the literature, namely $(\mathbb{R}_+\times W, dr^2+r^2 p^*g_W)$ via the change of variable $r:= e^{-t}$.

We have the following result:

\begin{proposition}\label{A5coneisom}
For any complete Riemannian manifold $(W,g_W)$, each homothety of the Riemannian cone $(\mathcal{C}(W)=\mathbb{R}\times W, 4e^{-2t}(dt^2+p^*g_W))$ 
is of the form $(t,w)\mapsto (t+\rho, \psi (w))$, where $e^{-2\rho}$ is the  dilatation factor and $\psi$ is an isometry of $(W,g_W)$.  In particular, all the isometries of the cone $\mathcal{C}(W)$ come from isometries of $W$. 
\end{proposition}

Proposition~\ref{A5coneisom} is proved in \cite{A5GOPP2006} for the compact case, \emph{i.e.} when $(W, g_W)$ is a compact Riemannian manifold. In \cite{A5BM2015}, F.~Belgun and A.~Moroianu extended this result to the larger class of so-called cone-like manifolds. For the complete case,  Proposition~\ref{A5coneisom} can be proven as follows, after making, for convenience, the coordinate change $r=2e^{-t}$. It is known that the metric completion of the cone $(\mathbb{R}_+\times W, dr^2+r^2 p^*g_W)$ is a metric space obtained by adding a single point. It can further be showed that the incomplete geodesics of the cone are exactly its rays, $r\mapsto (r,w)$, for any fixed $w\in W$. Since the image of an incomplete geodesic through any isometry is again an incomplete geodesic, it follows that each isometry $\varphi$ of the cone preserves the radial vector field, \emph{i.e.} $\varphi_* \left(\frac{\partial}{\partial r}\right)=\frac{\partial}{\partial r}$. For any vector field $X$ tangent to $W$, $\varphi_* (X)$ is also tangent to $W$, since it is orthogonal to $\varphi_* \left(\frac{\partial}{\partial r}\right)$ and $\varphi_* \left(\frac{\partial}{\partial r}\right)=\frac{\partial}{\partial r}$. Altogether, this implies that $\varphi$ has the following form: $(r,w)\mapsto (r, \psi(w))$, with $\psi$ an isometry of $W$. The statement for the homotheties then also follows, by composing for instance any isometry with the following homothety of the cone: $(r,w)\mapsto (e^{-\rho} r, w)$, for some $\rho\in\mathbb{R}$.

\begin{definition}
A {\it Sasaki structure} on a Riemannian manifold $(W,g_W)$ is a complex structure $J$ on its cone $\mathcal{C}(W)$, such that $(g_{cone}, J)$ is K\"ahler and for all $\lambda\in\mathbb{R}$, the homotheties $\rho_{\lambda} \colon \mathcal{C}(W)\to \mathcal{C}(W)$, $\rho_\lambda(t,w):=(t+\lambda, w)$ are holomorphic.
\end{definition}

There are many equivalent definitions of a Sasaki structure, see \emph{e.g.} the monography \cite{A5BG2008}. In particular, a Sasaki manifold $W^{2n+1}$ is endowed with a metric $g_W$ and  a contact $1$-form $\eta$ (\emph{i.e.} $\eta\wedge(\dd \eta)^n\neq 0$), whose Reeb vector field $\xi$, which is defined by the equations $\eta(\xi)=1$ and $\iota_\xi \dd \eta=0$, is Killing. The restriction of the endomorphism $\Phi:=\nabla^{g_W} \xi$ to the contact distribution $\ker(\eta)$ defines an integrable $CR$-structure. When needed, we may write a Sasaki structure by denoting all these data, as $(W, g_W, \xi, \eta, \Phi)$. We recall also that an automorphism of a Sasaki manifold is an isometry which preserves the Reeb vector field (see \emph{e.g.} \cite[Lemma~8.1.15]{A5BG2008}).  

\begin{definition}
A {\it toric Sasaki manifold} is a  Sasaki manifold $(W^{2m+1},g_W,\xi)$ 
endowed with an effective action of the torus
$T^{m+1}$, which
preserves the Sasaki structure and such that
$\xi$ is an element of the Lie algebra $\mathfrak{t}_{m+1}$ of the torus.
Equivalently, a toric Sasaki manifold is a Sasaki manifold
whose K\"ahler cone is a toric K\"ahler manifold.
\end{definition}

\noindent Thus, toric Sasaki manifolds are in particular contact toric manifolds of Reeb type.  E.~Lerman classified compact connected contact toric manifolds, see \cite[Theorem 2.18]{A5L2003}.\\

One way to construct Vaisman manifolds, \emph{cf.} R.~Gini, L.~Ornea, M.~Parton, \cite[Proposition 7.3]{A5GOP2005}, is the following:
\begin{proposition}
Let $W$ be a Sasaki manifold and $\Gamma$ be a group of biholomorphic homotheties of the K\"ahler cone $\mathcal{C}(W)$, acting freely and properly discontinuously on $\mathcal{C}(W)$, such that $\Gamma$ commutes with the radial flow generated by $\frac{\partial}{\partial t}$. Then $M:=\mathcal{C}(W)/\Gamma$ has a naturally induced Vaisman structure.
\end{proposition}

Conversely, it is implicitly proved in the work \cite{A5Vaisman1979} of I.~Vaisman that the universal covering of a Vaisman manifold, endowed with the K\"ahler metric, is a cone over a Sasaki manifold. R.~Gini, L.~Ornea, M.~Parton and P.~Piccinni showed in \cite{A5GOPP2006} that this result is true for any presentation of a Vaisman manifold. Let us recall that a \emph{presentation} of a locally conformally K\"ahler manifold is a pair $(K,\Gamma)$, where $K$ is a homothetic K\"ahler manifold, \emph{i.e.} the K\"ahler metric is defined up to homotheties, and $\Gamma$ is a discrete group of biholomorphic homotheties acting freely and properly discontinuously on $K$.  To any presentation, there is a group homomorphism, which associates to each homothety, its dilatation factor, $\rho_K\colon\Gamma\to \mathbb{R}^+$, such that $\gamma^*(g_K)=\rho_K(\gamma)g_K$, for any $\gamma\in\Gamma$, where $g_K$ is the K\"ahler metric (up to homothety) on $K$. The \emph{maximal presentation} is the universal covering $(\widetilde M, \pi_1(M))$ of the lcK manifold $M$ and the \emph{minimal presentation} or \emph{minimal covering}  is given by 
$$(\widehat M:=\widetilde M/(\mathrm{Isom}(\widetilde M, g_K)\cap \pi_1(M)),\Gamma_{\mathrm{min}}:= \pi_1(M)/(\mathrm{Isom}(\widetilde M, g_K)\cap \pi_1(M)).$$ Hence, each $\gamma\in\Gamma_{\mathrm{min}}\setminus\{\mathrm{id}\}$ acts as a proper homothety on $(\widehat M, g_K)$, \emph{i.e.} $\gamma$ is not an isometry: $\rho_{\widehat{M}}(\gamma)\neq 1$. We can now state the following consequence of \cite[Theorem 4.2]{A5GOPP2006}:

\begin{proposition} \label{A5univcover}
The minimal (resp. universal) covering of a Vaisman manifold $(M,g,J,\theta)$ is biholomorphic and conformal to the K\"ahler cone of a Sasaki manifold and $\Gamma_{\mathrm{min}}$ (resp. $\pi_1(M)$) acts on the cone by biholomorphic homotheties with respect to the K\"ahler cone metric.
\end{proposition}

\begin{remark}\label{A5thminexact}
On the minimal covering $\widehat M$ of an lcK manifold $(M,g,J,\theta)$, the pull-back $\hat{\theta}$ of the Lee form is exact. This property is clearly true on the universal covering $\widetilde M$, since the pull-back of $\theta$ is still closed, hence exact: $\tilde\theta=\mathrm{d} f$, as $\widetilde M$ is simply-connected. The minimal covering is obtained from the universal covering by quotiening out the isometries (of the K\"ahler metric $e^{-f}\tilde g$ on $\widetilde M$) in $\pi_1(M)$. Therefore, the function $f$ projects onto a function $\hat f\in\mathcal{C}^\infty(\widehat M)$, such that $\hat\theta=\mathrm{d} \hat f$. 
\end{remark}

\section{Twisted Hamiltonian Actions on lcK manifolds}

Let $(M,g,J,\theta)$ be an lcK manifold. We consider $\mathrm{d}^\theta$, the so-called \emph{twisted differential}, defined by:
\[\mathrm{d}^\theta\colon \Omega^*(M)\longrightarrow \Omega^{*+1}(M), \quad \mathrm{d}^\theta \alpha:=d\alpha-\theta\wedge\alpha.\]
Remark that $\mathrm{d}^\theta\circ \mathrm{d}^\theta=0$ if and only if $\mathrm{d}\theta=0$ and that in this case $\mathrm{d}^\theta$ anti-commutes with $d$. By definition, on an lcK manifold, we have $\mathrm{d}^\theta\omega=0$.

Let $\mathcal{L}^\theta$ denote the \emph{twisted Lie derivative} defined by
$\mathcal{L}^\theta_X:=\mathrm{d}^\theta\circ\iota_X+\iota_X\circ \mathrm{d}^\theta$.
The following relation holds, for any $X\in\mathfrak{X}(M)$:
\begin{equation}\label{A5lietheta}
\mathcal{L}^\theta_X\omega=\mathcal{L}_X\omega-\theta(X)\omega,
\end{equation}
as follows from the following direct computation:
\begin{equation*}
 \begin{split}
  \mathcal{L}_{X}\omega&=d\iota_{X}\omega+\iota_{X}d\omega= d\iota_{X}\omega+\iota_{X}(\theta\wedge \omega)\\
  &=d\iota_{X}\omega+\theta(X)\omega-\theta\wedge\iota_{X}\omega= \mathrm{d}^\theta(\iota_{X}\omega)+\theta(X)\omega\\
  &=\mathcal{L}^\theta_{X}\omega+\theta(X)\omega.
 \end{split}
\end{equation*}

\begin{definition}
Given a function $f\in C^\infty(M)$ on an lcK manifold $(M,g,J,\theta)$,  its \emph{associated twisted Hamiltonian vector field} $X_f$ is defined as the $\omega$-dual of $\mathrm{d}^\theta f=d f -f\theta$, \emph{i.\,e.} $\iota_{X_f}\omega=\mathrm{d}^\theta f$.
The subset $\mathrm{Ham}^\theta(M)\subset\mathfrak{X}(M)$ of \emph{twisted Hamiltonian vector fields} is that of vector fields on $M$ admitting such a presentation.
\end{definition}

We remark  that the notion of twisted Hamiltonian vector field is invariant under conformal changes of the metric, even though the function associated to a twisted Hamiltonian vector field changes by the conformal factor. More precisely, if $g'=e^{\alpha}g$, then $\omega'=e^\alpha\omega$, $\theta'=\theta+d\alpha$ and the following relation holds $
 d^{\theta'} f=e^{\alpha}\mathrm{d}^{\theta}(e^{-\alpha}f)$. Note that if $M$ is not globally conformally K\"ahler, the map $C^\infty(M) \to \mathfrak{X}(M)$, $f\mapsto X_f$ is injective.

\begin{lemma} \label{A5twistedham}
Twisted Hamiltonian vector fields on an lcK manifold $(M,g,J,\theta)$ have the following properties: 
\begin{itemize}
\item[(i)] $\forall X\in\mathrm{Ham}^\theta(M)$ the
following relations hold:
\begin{equation*}
\mathcal{L}^\theta_X \omega=\mathrm{d}^\theta \iota_{X}\omega=0.
\end{equation*}
\item[(ii)] $\mathrm{Ham}^\theta(M)$ is a vector subspace of $\mathfrak{X}(M)$, such that
$J\theta^\sharp\in\mathrm{Ham}^\theta(M)$ and $\theta^\sharp\notin \mathrm{Ham}^\theta(M)$.
\item[(iii)] Twisted Hamiltonian vector fields
do not leave invariant the fundamental form $\omega$, but conformally invariant, i.e. for all $X\in \mathrm{Ham}^\theta(M)$:
\begin{equation}\label{A5tH}
  \mathcal{L}_{X}\omega=\theta(X)\omega.
\end{equation}
\end{itemize}
\end{lemma}

\begin{proof}
(i) follows directly by the Cartan formula.\\
(ii) This can be seen as follows:
 \[(\iota_{J\theta^\sharp}\omega)(X)=\omega(J\theta^\sharp,X)=g(J\theta^\sharp,JX)=\theta(X),\]
 hence $\iota_{J\theta^\sharp}\omega=\theta=\mathrm{d}^\theta f$, where $f$ is the constant function equal to $-1$.
 
However, the Lee vector field $\theta^\sharp$ is not twisted Hamiltonian, since the $1$-form $\iota_{\theta^\sharp}\omega$ is not even $\mathrm{d}^\theta$-closed:
 \[\mathrm{d}^\theta \iota_{\theta^\sharp}\omega=\mathcal{L}^{\theta}_{\theta^\sharp}\omega-\iota_{\theta^\sharp}\mathrm{d}^\theta\omega=\mathcal{L}_{\theta^\sharp}\omega-\theta(\theta^\sharp)\omega=-\omega\neq 0\]
(iii) follows from \eqref{A5lietheta}.
\end{proof}

The \emph{twisted Poisson bracket} on $C^\infty(M)$ is defined by:
\begin{equation}\label{A5Pb}
\{f_1, f_2\}:=\omega((\mathrm{d}^\theta f_1)^\sharp, (\mathrm{d}^\theta f_2)^\sharp)=\omega(X_{f_1}, X_{f_2}).
\end{equation}
and it turns $C^\infty(M)$ into a Lie algebra. Hence, the following equality holds: $X_f=-J((\mathrm{d}^\theta f)^{\sharp})$.

Recall that for any action of a Lie group $G$ on a manifold $M$, and for any $X$ in the Lie algebra of $G$, it is naturally associated the so-called {\it fundamental vector field} $X_M$, defined as $X_M(x)=\frac{d}{dt}|_{t=0}(\exp(tX)\cdot x)$, for any $x\in M$, where $\exp\colon \mathfrak{g}\to G$ is the exponential map of the group $G$ and $"\cdot"$ denotes the action of $G$ on $M$.
We identify the elements of the Lie algebra with the induced fundamental vector fields, when  there is no ambiguity.

\begin{definition}
 Let $(M,g,J,\theta)$ be a  locally conformally K\"ahler manifold with fundamental $2$-form $\omega$. The action of a Lie group $G$ on $M$ is called 
 \begin{itemize}
  \item \emph{weakly twisted Hamiltonian} if the associated fundamental vector fields are twisted Hamiltonian, \emph{i.\,e.} there exists a linear map 
 \[\mu:\mathfrak{g}\to C^\infty(M)\]
 such that $\iota_X \omega=\mathrm{d}^\theta \mu^X$, for all fundamental vector fields $X\in\mathfrak{g}$. This means that the  twisted Hamiltonian vector field associated to $\mu^X:=\mu(X)$ is exactly the vector field $X$. The condition here is that all fundamental vector fields are twisted Hamiltonian: $\mathfrak{g}\subseteq \mathrm{Ham}^\theta(M)$.
 \item \emph{twisted Hamiltonian} if the map $\mu$ can be chosen to be a Lie algebra homomorphism with respect to the Poisson bracket defined in \eqref{A5Pb}.\\
 In this case the Lie algebra homomorphism $\mu$ is called a \emph{momentum map} for the action of $G$.
  \end{itemize}
The map $\mu$ may  equivalently be considered as the map:
 \[\mu\colon M\to \mathfrak{g}^*, \quad \<\mu(x),X\>:=\mu^X(x), \;\forall X\in\mathfrak{g}, \forall x\in M.\]

\end{definition}

\noindent The condition on $\mu\colon \mathfrak{g}\to C^\infty(M)$ to be a homomorphism of Lie algebras is equivalent to $\mu\colon M\to \mathfrak{g}^*$ being equivariant with respect to the adjoint action on $\mathfrak{g}^*$, the dual of the Lie algebra of $G$.

\begin{remark}
The property of an action to be twisted Hamiltonian is a property of the conformal structure, even though the Poisson structure on $\mathcal{C}^\infty(M)$ is not conformally invariant. If $g'=e^{\alpha}g$, then $\mu^{\theta'}=e^{\alpha}\mu^{\theta}$.
\end{remark}

We define toric lcK manifolds by analogy  to other toric geometries, as follows:

\begin{definition}\label{A5deftoriclck}
 A connected locally conformally K\"ahler manifold $(M,[g],J)$ of dimension $2n$  equipped with an effective holomorphic and twisted Hamiltonian  action of the standard (real) $n$-dimensional torus $T^{n}$:
 \[\tau: T^{n}\to \mathrm{Diff}(M)\]
 is called a \emph{toric locally conformally K\"ahler manifold}.
\end{definition}

\begin{remark}
Let $(M,g,J,\theta)$ be a Vaisman manifold. Let $\pi:\widetilde{M}\to M$ be the universal covering of $M$. We still denote by $J$ the induced complex structure on $\widetilde{M}$. 
Since $\pi^*\theta$ is closed and $\widetilde M$ simply-connected, it follows that it is also exact, \emph{i.e.} there exists $h\in\mathcal{C}^{\infty}(\widetilde M)$ such that $\pi^*\theta=d h$. We define the following metric and associated fundamental form on $\widetilde{M}$, which build a K\"ahler structure:
\[\tilde{g}:=e^{-h}\pi^* g, \quad \widetilde{\omega}=e^{-h}\pi^*\omega.\]
This can be showed as follows:
\[d\widetilde{\omega}=d(e^{-h}\pi^*\omega)=e^{-h}(-dh\wedge \pi^*\omega+\pi^*d\omega)=e^{-h}(-dh\wedge \pi^*\omega+\pi^*\theta\wedge \pi^*\omega)=0.\]

A twisted Hamiltonian $G$-action on $(M,g,J,\omega)$ is equivalent to a Hamiltonian $\widetilde{G}^0$-action on the K\"ahler manifold $(\widetilde{M},\tilde{g}, J,\widetilde{\omega})$, where $\widetilde{G}^0$ denotes the identity component of the universal covering of $G$. To check this it is sufficient to consider the infinitesimal action of the Lie algebra $\mathfrak{g}$. We identify $X\in\mathfrak{g}=\mathrm{Lie}(G)$ with its associated fundamental vector field on $M$. Then, the fundamental vector field associated to $\mathfrak{g}=\mathrm{Lie}(\widetilde{G})$ equals $\pi^* X$. If $\mu:\mathfrak{g}\to \mathcal{C}^{\infty}(M)$ is the momentum map of the twisted Hamiltonian action of $G$ on $M$, then 
the map
\[\widetilde\mu:\mathfrak{g}\to \mathcal{C}^{\infty}(\widetilde{M}), \quad X \mapsto e^{-h}\pi^* \mu^X\]
is the momentum map of the Hamiltonian action of $\widetilde{G}$ on $\widetilde{M}$ with respect to the K\"ahler form $\widetilde{\omega}$. In fact, for any $X\in\mathfrak{g}$, we have $\iota_X\omega=\mathrm{d}^{\theta} \mu^X$, by definition of a twisted Hamiltonian action.
We now compute on $\widetilde{M}$:
\begin{equation*}
 \begin{split}
  \iota_{\pi^*X}\widetilde{\omega}&=\iota_{\pi^*X}(e^{-h}\pi^*\omega)=e^{-h}\pi^*(\iota_X\omega)=e^{-h}\pi^*(\mathrm{d}^{\theta} \mu^X)\\
  &=e^{-h}\pi^*(d\mu^X-\mu^X\theta)=e^{-h}(d(\pi^*\mu^X)-\pi^*\mu^X \cdot d h)=d(e^{-h}\pi^*\mu^X).
 \end{split}
\end{equation*}
Since $\mu$ is a homomorphism of Poisson algebras, the same holds for $\widetilde\mu$.
\end{remark}

\section{Toric Vaisman manifolds}

In this section $(M,g,J,\theta)$ denotes a Vaisman manifold.

\begin{remark}
 Let $X$ be a holomorphic Killing vector field on a Vaisman manifold $(M,J,g, \theta)$.
The $1$-form $\iota_X\omega$ is $\mathrm{d}^{\theta}$-closed iff $\mathcal{L}^\theta_X\omega=0$.
 Since by assumption $\mathcal{L}_X\omega=0$,
 it follows from \eqref{A5lietheta} that $\mathrm{d}^{\theta}\iota_X\omega=0$ iff $\theta(X)=0$.
\end{remark}

\begin{lemma}\label{A5rtw}
Any twisted Hamiltonian holomorphic action of a Lie group $G$ on a complete Vaisman manifold $(M,g,J,\theta)$ is {\it automatically isometric with respect to the Vaisman metric} and the following inclusion holds:
 \[\mathfrak{g}\subseteq\mathfrak{isom}(M)\cap \ker(\theta).\]
 \end{lemma}
 \begin{proof}
 By assumption we have:
 \[\mathfrak{g}\subseteq \mathfrak{hol}(M) \cap \mathrm{Ham}^\theta(M).\]
 Let $X\in\mathfrak{g}\subseteq \mathrm{Ham}^\theta(M)$. From \eqref{A5tH}, it follows that $\mathcal{L}_X\omega=\theta(X)\omega$. Since $X$ is a holomorphic vector field, we have $\mathcal{L}_XJ=0$. From the following relation between the Lie derivatives:
 \[(\mathcal{L}_{X}\omega)(\cdot,\cdot)=(\mathcal{L}_X g)(\cdot, J\cdot)+g(\cdot, (\mathcal{L}_X J)\cdot), \quad \forall X\in TM,\]
it follows that $\mathcal{L}_X g=\theta(X)g$. Hence, $X$ is a conformal vector field. 

In order to show that $X$ is a Killing vector field, we consider the universal covering $\widetilde M$ of $M$, which by Proposition~\ref{A5univcover}, is biholomorphic conformal to the K\"ahler cone $(\mathcal{C}(\widetilde W), 4e^{-2t}(\dd t^2+p^*g_{\widetilde W}))$ over a Sasaki manifold $(\widetilde W, g_{\widetilde W})$. The lift of $X$ to $\widetilde M$, denoted by $\widetilde X$,  is a conformal Killing vector field with respect to the pull-back metric $\tilde g:=\pi^*g$, where $\pi$ denotes the natural projection from $\widetilde M$ to $M$. Moreover, we claim that $\widetilde X$ is a Killing vector field with respect to the K\"ahler cone metric $g_K$, as the following computation shows:
\begin{equation*}
\begin{split}
\mathcal{L}_{\widetilde X} g_K&=\mathcal{L}_{\widetilde X} (e^{-2t}\tilde g)=\widetilde{X}(e^{-2t})\tilde g+\mathcal{L}_{\widetilde X} (e^{-2t}\tilde g)=e^{-2t}(-2\dd t(\widetilde X)\tilde g+\pi^*(\mathcal{L}_{X} g))\\
&=e^{-2t}(-2\dd t(\widetilde X)+\pi^*(\theta(X)g))=0,
\end{split}
\end{equation*} 
where we used that $\pi^*\theta=2\dd t$.

Let us note, that by the theorem of Hopf-Rinow, the assumption on $(M,g)$ to be complete implies that $(\widetilde M, \tilde g)$ is complete and further that also $(\widetilde W, g_{\widetilde W})$ is complete.
Thus, applying Proposition~\ref{A5coneisom}, it follows that all Killing vector fields of the cone metric on $\mathcal{C}(\widetilde W)$ are projectable onto Killing vector fields of the Sasaki metric $g_{\widetilde W}$ on $\widetilde W$. In particular, it follows that $\left[\widetilde X,\frac{\partial}{\partial t}\right]=0$. This implies that $[X, \theta^{\sharp}]=0$.

From $\mathcal{L}_X g=\theta(X)g$ and $[X, \theta^{\sharp}]=0$, it follows:
$$\theta(X)\theta=\mathcal{L}_X \theta=\mathrm{d}(\theta(X)),$$
where the last equality is obtained by Cartan formula. This identity applied to $\theta^{\sharp}$ yields $\theta^{\sharp}(\theta(X))=\theta(X)$. 
On the other hand, we compute:
\[\theta^\sharp(\theta(X))=(\mathcal{L}_{\theta^\sharp}\theta)(X)+\theta([\theta^\sharp, X])=0,\]
since again by the Cartan formula, we have $\mathcal{L}_{\theta^\sharp}\theta=\mathrm{d}\iota_{\theta^\sharp}{\theta}=0$. Therefore, we obtain $\theta(X)=0$. Hence,  $X$ is a Killing vector field of $g$ and is orthogonal to $\theta^\sharp$.
\end{proof}

\begin{remark}\label{A5remal}
On a compact Vaisman manifold, A.~Moroianu and L.~Ornea~\cite{A5MO2009} proved that  every conformal Killing vector field is a Killing vector field with respect to the Vaisman metric. Recently, P.~Gauduchon and A.~Moroianu~\cite{A5Gauduchon-Moroianu2016} proved the following more general result: on a connected compact oriented Riemannian manifold admitting a non-trivial parallel vector field, any conformal Killing vector field is Killing.
\end{remark}

The next result shows that an action  preserving the whole Vaisman structure is automatically twisted hamiltonian. Moreover, it shows that the momentum map is given by  the anti-Lee $1$-form. More precisely, we have:

\begin{lemma}
Let $(M,g,J,\theta)$ be a Vaisman manifold and $X$ be a holomorphic Killing vector field on $M$, which is in the kernel of $\theta$. Then $X$ is a twisted Hamiltonian vector field with Hamiltonian function $f:=J\theta(X)$, \emph{i.e.} the following equality holds:
$$\iota_X\omega={\rm d}f-f\theta={\rm d}^\theta f.$$ 
\end{lemma}

\begin{proof}
We compute the differential of $f$ as follows:
\begin{equation}\label{A5iX omega}
 \textrm{d}f(Y) =(\nabla_YJ)(\theta\wedge X)+g( J\theta,\nabla_YX)
 =\frac{1}{2}(f\theta(Y)+\omega(X,Y))-g(Y ,\nabla_{J\theta}X)
\end{equation}
since $\theta (X)=0$, $|\theta|=1$ and $\nabla X$ is skew-symmetric. Note that 
$$-\nabla_{J\theta}X=\mathcal{L}_XJ\theta-\nabla_XJ\theta=-\nabla_XJ\theta=\frac{1}{2}
(f\theta+JX),$$ since $X$ 
preserves $\theta$ and $J$, and $\theta$ is parallel. Substituting in \eqref{A5iX omega}, we obtain
$$\textrm{d}f= \frac{1}{2}(f\theta+\iota_X\omega)-\nabla_{J\theta}X=f\theta+\iota_X\omega.$$\end{proof}

We consider now toric Vaisman manifolds, as a special class of toric lcK manifolds, \emph{cf.} Definition~\ref{A5deftoriclck}. 

\begin{remark}\label{A5torvaisisom}
 As in the case of K\"ahler toric manifolds, where a Hamiltonian holomorphic action automatically preserves the K\"ahler metric, also on Vaisman manifolds a twisted Hamiltonian holomorphic action preserves the Vaisman metric, see Lemma~\ref{A5rtw}.
\end{remark}

\begin{remark}\label{A5maxdim}
The maximal dimension of an effective twisted Hamiltonian torus action on a $2n$-dimensional Vaisman manifold is $n$. The proof follows by the same argument as for Hamiltonian actions on symplectic manifolds (for details see \emph{e.g.} \cite{A5Dasilva2001}). 
\end{remark}

In particular,  on a toric Vaisman manifold, by Lemma~\ref{A5rtw}, the following inclusions hold: 
\[\mathfrak{t}_n\subseteq\mathfrak{hol}(M)
\cap \mathrm{Ham}^\theta(M)\subseteq\mathfrak{isom}(M)\cap (\theta^\sharp)^\perp,\]
where $\mathfrak{t}_n$ denotes the Lie algebra of $T^n$.
Moreover, we show the following:

\begin{lemma}\label{A5Jthtor}
On a toric Vaisman manifold $(M^{2n},g,J, \theta)$, the anti-Lee vector field is part of the twisted Hamiltonian torus action, \emph{i.e.} $J\theta^{\sharp}\in\mathfrak{t}_n$.
\end{lemma}

\begin{proof}
We argue by contradiction and assume that $J\theta^{\sharp}\notin\mathfrak{t}_n$. By Remark~\ref{A5maxdim}, $n$ is the maximal dimension of a torus acting effectively and twisted Hamiltonian on a  $2n$-dimensional Vaisman manifold. Hence, it suffices to show that $\mathfrak{t}_n\cup\{J\theta^{\sharp}\}$ is still an abelian Lie algebra acting twisted Hamiltonian on $(M,g,J)$. This is a consequence of Lemma~\ref{A5twistedham}, (ii), stating that $J\theta^{\sharp}\in \mathrm{Ham}^{\theta} M$, and of Lemma~\ref{A5thJth}, where it is shown that $J\theta^{\sharp}$ is in the center of $\mathfrak{aut}(M)$, so in particular commutes with all elements of $\mathfrak{t}_n\subset\mathfrak{aut}(M)$.
\end{proof}

\begin{example}\label{A5example}
The standard example of a Vaisman manifold is $S^1\times S^{2n-1}$, endowed with 
the complex structure and metric induced by the diffeomorphism
$$\Psi\colon\mathbb C^n\setminus\{0\}/\mathbb Z\longrightarrow S^1\times S^{2n-1},\;\; 
[z]\longmapsto (e^{i\ln |z|}, \frac{z}{|z|}),$$
where $[z]=[z']$ if and only if there exists $k\in\mathbb Z$ such that $z'=e^{2\pi k}z$. 
The Hermitian metric $\frac{{\rm d}z_j\otimes {\rm d}\bar z_j}{|z|^2}$ on 
$\mathbb C^n\setminus\{0\}$ is invariant under the action of $\mathbb Z$ and hence it descends to 
a Hermitian metric on $\mathbb C^n\setminus\{0\}/\mathbb Z$, with   
$$g:={\rm Re}\frac{{\rm d}z_j\otimes {\rm d}\bar z_j}{|z|^2},\quad 
\omega:=i\frac{{\rm d}\bar z_j\wedge {\rm d} z_j}{|z|^2},\quad 
\theta:={\rm d}\ln|z|^{-2}$$
defining the lcK metric, $2$-fundamental form and Lee form respectively. The Lee form $\theta$
is parallel, hence $S^1\times S^{2n-1}$ is Vaisman. We define the following $T^n$ action on 
$\mathbb C^n\setminus\{0\}/\mathbb{Z}\simeq S^1\times S^{2n-1}$: $t\longmapsto ([z]\mapsto[t_1z_1,\ldots, t_nz_n])$. It is easy to 
check that it is effective and holomorphic. A basis of fundamental vector fields of the action is given by 
$X_j([z]):=iz_j\frac{\partial}{\partial z_j}-i\bar z_j\frac{\partial}{\partial \bar z_j}$, for $j\in
\{1,\ldots, n\}$. Therefore, we have 
$$\iota_{X_j}\omega=\frac{\bar z_j {\rm d} z_j+z_j {\rm d}\bar z_j}{|z|^2}=
{\rm d}\left(\frac{|z_j|^2}{|z|^2}\right)-
\frac{|z_j|^2}{|z|^2}\theta={\rm d}^\theta \left(\frac{|z_j|^2}{|z|^2}\right),$$ 
which shows that the action is also twisted Hamiltonian, with momentum map 
$\mu^{X_j}([z]):=\frac{|z_j|^2}{|z|^2} $. Hence, $S^1\times S^{2n-1}$ is a toric Vaisman manifold.
\end{example}

We are now ready to state our result:

\begin{theorem}\label{A5vaissas}
Let $(M^{2n},J,g,\theta)$ be a complete Vaisman manifold and $W^{2n-1}$ be its associated Sasaki manifold, such that the minimal covering $\widehat M$ of $M$ is biholomorphic and conformal to the K\"ahler cone over $W$. If $M$ is a toric Vaisman manifold, then $W$ is a toric Sasaki manifold.
\end{theorem}

\begin{proof}
Let $\pi\colon\widehat M\to M$ denote the projection of the minimal covering. $\widehat M$ is naturally endowed with the pullback metric $\hat g:=\pi^*g$ and the complex structure $\widehat J$, such that $\pi$ is a local isometric biholomorphism. 
By Proposition~\ref{A5univcover}, we know that $ \widehat M=\mathcal{C}(W)$, where $(W,g_W,\xi, \eta, \Phi)$ is a complete Sasaki manifold and $\mathcal{C}(W)=\mathbb{R}\times W$ is its cone. 
By this identification, the lift of the $1$-form $\theta$, which is exact on $\widehat M$ \emph{cf.} Remark~\ref{A5thminexact}, equals $2\mathrm{d}t$, where $t$ is the coordinate of the factor $\mathbb{R}$. Moreover, the relationship between the induced metric $\hat g$ and the K\"ahler cone metric $g_K:=g_{\mathrm{cone}}=4e^{-2t}(dt^2+p^*g_W)$ is the following: 
$\hat g=e^{2t}g_K=4(\mathrm{d} t^2+p^*g_W)$. The complex structure $\widehat J$ coincides with the usual complex structure induced by the Sasaki structure on the K\"ahler cone, \emph{i.e.} $\widehat J (p^*\xi)=\frac{\partial}{\partial t}$ and $\widehat J(\frac{\partial}{\partial t})=-p^*\xi$ and $\widehat J$ coincides with the transverse complex structure $\Phi$ on the horizontal distribution $\xi^\perp$ in $\mathrm{T}W$.

By assumption, the Vaisman manifold $M^{2n}$ is toric, hence it is equipped with an effective twisted Hamiltonian action $\tau: T^{n}\to \mathrm{Diff}(M)$. 

 By Lemma~\ref{A5Jthtor}, $J\theta^{\sharp}\in\mathfrak{t}_n$, hence we may choose  a basis $X_1, \dots, X_n$ of the Lie algebra $\mathfrak{t}_n\cong \mathbb{R}^{n}$, such
 that the fundamental vector field on $M$ associated to $X_1$ is $-2 J\theta^{\sharp}$. For simplicity, we still denote the induced fundamental vector fields of the action by $X_j\in \mathfrak{X}(M)$, for $1\leq j\leq n$. 
From Remark~\ref{A5omth} and Lemma~\ref{A5rtw}, it follows that $[X_j, \theta^{\sharp}]=0$, for all $j$.
Moreover, Lemma~\ref{A5rtw} implies that $g(X_j, \theta^{\sharp})=0$,  for all $j$. 

Since $\pi$ is a local diffeomorphism, we can lift each vector field $X_j$ uniquely to a
vector field $\widehat {X_j}$ on $\widehat M$. As $X_j$ commute pairwise, the same is true for their lifts, \emph{i.e.} we have $[\widehat {X_j},\widehat {X_k}]=0$, for all $1\leq j,k\leq n$.  Note that the lift of $\theta^{\sharp}$ to $\widehat M$ equals $(\pi^*\theta)^{\sharp,\hat g}$, \emph{i.e.} the dual of $\pi^*\theta=2\mathrm{d} t$ with respect to the metric $\tilde g$. Thus, $\widehat {\theta^{\sharp}}=\frac{1}{2}\frac{\partial}{\partial t}$. Hence, $\widehat {X_1}$, which is the lift to the universal covering of $-2J\theta^{\sharp}$, equals $\widehat {X_1}=-2\widehat J ((\pi^*\theta)^{\sharp,\hat g})=-2\widehat J (\frac{1}{2} \frac{\partial}{\partial t})=p^*\xi$. It is then clear, that $\widehat {X_1}$ projects through $p$ on $W$ to $\xi$. In fact, each lift $\widehat {X_j}$ is projectable onto vector fields on $W$, since they are constant along~$\mathbb{R}$: $\left[\frac{\partial}{\partial t}, \widehat {X_j}\right]=2\left[\widehat {\theta^{\sharp}}, \widehat {X_j}\right]=2\widehat {[\theta^{\sharp}, X_j]}=0$, for all $1\leq j\leq n$. We denote their projections by $Y_j:=p_*\widehat {X_j} \in\mathfrak{X}(W)$, for all $1\leq j\leq n$. In particular, $Y_1$ equals the Reeb vector field $\xi$. We remark that the vector fields $\{Y_1, \dots, Y_n\}$ commute pairwise and are complete. Thus, they give rise to an effective action $\tau_W\colon\mathbb{R}^n\to \mathrm{Diff}(W)$.

We show next that this action acts by automorphisms of the Sasaki structure, \emph{i.e.} the vector fields $\{Y_1, \dots, Y_n\}$ preserve $(g_W,\xi, \eta, \Phi)$. More precisely, it is sufficient to show that they are Killing vector fields, since they commute with $Y_1=\xi$. The torus action $\tau$ on the Vaisman manifold is by 
isometries, \emph{cf.} Lemma~\ref{A5rtw}, so each vector field $X_j$ is Killing with respect to the Vaisman metric $g$. The projection $\pi$ being a local isometry between $(\widehat  M, \hat  g)$ and $(M,g)$, it follows that each lift $\widehat {X_j}$ is a Killing vector field with respect to the metric $\hat g=4(\mathrm{d} t^2+p^*g_W)$, hence their projections onto $(W,g_W)$ are still Killing vector fields with respect to the Sasaki metric $g_W$.  

In order to conclude that $W$ is a toric Sasaki manifold, we need to argue why the action $\tau_W$ of $\mathbb{R}^n$ on $W$ naturally induces an action of the torus $T^n$ on $W$. For this, it is enough to show that for each $X\in\mathfrak{t}_n$ with the associated fundamental vector on $M$ having closed orbits of the same period, also its lift $\widehat X$ has the same property.  Let $\Psi_s$ and $\widehat \Psi_s$ denote the flow of $X$, resp. of $\widehat X$. We assume that all the orbits of $X$ close after time $s=1$, \emph{i.e.} $\Psi_1(x)=x$, for all $x\in M$.  Then, since  $\pi\circ \widehat \Psi_s=\Psi_s\circ\pi$, for all values of $s$, it follows that for each $(t, w)\in\widehat M$, we have  $\pi(( t, w))=\pi (\Psi_1(t,w))$, so there exists $\gamma\in\Gamma_{\mathrm{min}}$, such that $\Psi_1(t,w)=\gamma \cdot (t,w)$. A priori $\gamma$ may depend on the choice of $(t, w)$, but since the function defined in this way would be continuous with values in the discrete group $\Gamma_{\mathrm{min}}$, it must be constant. On the other hand, $\widehat X$ is a Killing vector field with respect to the K\"ahler metric $g_K=e^{-2t}\hat g$ (since, as shown above,  $\widehat X$ is Killing with respect to the pull-back $\hat g$ of the Vaisman metric and it also leaves invariant the conformal factor, as $\theta(X)=0$ implies $\mathrm{d}t(\widehat X)=0$). Applying Proposition~\ref{A5coneisom}, we have $\Psi_1(t,w)=(t, \psi(w))$, where $\psi$ is an isometry of $W$. Hence, for all $(t,w)\in\widehat M$, we have: $\gamma \cdot (t,w)=(t, \psi(w))$, in particular $\gamma$ acts as an isometry of the K\"ahler metric. As $\gamma\in \Gamma_{\mathrm{min}}$, it follows from the definition of $\Gamma_{\mathrm{min}}$ that $\gamma=\mathrm{id}$. We conclude that all orbits of $\widehat X$ are also closed with the same period.   
\end{proof}

\begin{remark}
The argument in the proof of Theorem~\ref{A5vaissas} also applies to the universal covering of a toric complete Vaisman manifold, showing that it carries an effective holomorphic Hamiltonian action of $\mathbb{R}^n$, \emph{i.e.} a completely integrable Hamiltonian system, with respect to its K\"ahler structure.
\end{remark}

Conversely, we show:
\begin{theorem}\label{A5sasvais}
Let $(W^{2n-1},g_W, \xi,\eta, \Phi)$ be a toric complete Sasaki mani\-fold with the effective torus action $\tau\colon T^n\to\mathrm{Diff}(W)$.  
Let $\Gamma$ be a discrete group of biholomophic homotheties acting freely and properly discontinuously on  the K\"ahler cone $(\mathcal{C}(W), 4e^{-2t}(dt^2+p^*g_W))$ and hence inducing a Vaisman structure on the quotient $M:=\mathcal{C}(W)/\Gamma$. If the  action $\tau$, which is extended to $\mathcal{C}(W)$ acting trivially on $\mathbb{R}$, commutes with $\Gamma$, then $M$ is a toric  Vaisman manifold.
\end{theorem}

\begin{proof}

Let $\pi$ denote the projection corresponding to the action of the group~$\Gamma$, $\pi\colon \mathcal{C}(W)\to M=\mathcal{C}(W)/\Gamma$. We still denote by $\tau$ the extension of the $T^n$-action on $\mathbb{R}\times W$, obtained by letting $T^n$ act trivially on $\mathbb{R}$. The assumption on this action to commute with the group $\Gamma$, ensures that it naturally projects onto a $T^n$-action on the quotient $M$, that we denote by $\bar\tau\colon T^n \to \mathrm{Diff}(M)$. In order to show that $M$ is a toric Vaisman manifold, we need to check that this action is effective, holomorphic and twisted Hamiltonian with respect to the induced Vaisman structure on $M^{2n}$, that we denote by $(g,J,\theta)$. For this, we study the induced fundamental vector fields of the action.

 We recall that through $\pi$, the exact form 
 $2\mathrm{d} t$ projects to the Lee form $\theta$ of the induced Vaisman structure on the quotient $M$, the vector field $\frac{1}{2}\frac{\partial}{\partial t}$ projects to the Lee vector field $\theta^{\sharp}$ and hence, $p^*\xi=-J(\frac{\partial}{\partial t})$ projects to $-2J\theta^{\sharp}$. We also know that the metric which projects onto the Vaisman metric $g$ is the product metric $4(\mathrm{d}t^2+g_W)$.
 
By assumption, the action $\tau\colon T^n\to\mathrm{Diff}(W)$ preserves the Sasaki structure and has the property that $\xi\in\mathfrak{t}_{n}$. We choose  a basis $X_1, \dots, X_n$ of the Lie algebra $\mathfrak{t}_n\cong \mathbb{R}^{n}$, such
 that the fundamental vector field on $M$ associated to $X_1$ is $-\frac{1}{2} \xi$. We still denote the induced fundamental vector fields of the action by $X_j\in \mathfrak{X}(W)$, for $1\leq j\leq n$. 
 We consider the pull-back of these vector fields on $\mathcal{C}(W)=\mathbb{R}\times W$ through the projection $p\colon \mathbb{R}\times W\to W$, $Y_j:=p^* X_j$, for $1\leq j\leq n$.  Hence, $Y_1=-p^*(\frac{1}{2}\xi)$ projects through $\pi$ to $J\theta^{\sharp}$, which is a Killing vector field with respect to the Vaisman metric $g$. We claim that this property is true for all vector fields $Y_j$. First we notice that each $Y_j$ is projectable through $\pi$, because of the hypothesis on the action of $T^n$ on $\mathbb{R}\times W$ to commute with $\Gamma$. We denote the projected vector fields by $\overline{Y_j}:=\pi_* Y_j$. In fact, by construction, the vector fields $\{\overline{Y_1},\dots, \overline{Y_n}\}$ are exactly the fundamental vector fields of the action $\bar\tau$ corresponding to the fixed basis $\{X_1, \dots, X_n\}$ of $\mathfrak{t}_n$.
As each $X_j$ is a Killing vector field with respect to the Sasaki metric $g_W$, it follows that $Y_j=p^*X_j$ is a Killing vector field with respect to the product metric $4(\mathrm{d}t^2+g_W)$ on $\mathbb{R}\times W$. Since $\pi$ is a local isometry between $(\mathbb{R}\times W, 4(\mathrm{d}t^2+g_W))$ and $(M,g)$, the vector fields $\overline{Y_j}$ are still Killing with respect to $g$.

The same argument works in order to show that $\overline{Y_j}$ is a holomorphic vector field on the Vaisman manifold $(M,g,J,\theta)$.  Namely, we use that $\pi$ is a local biholomorphism and that $Y_j=p^*X_j$ is a holomorphic vector field with respect to the induced holomorphic structure $J_K$ on the cone $\mathcal{C}(W)$. The last statement follows from the following straightforward computation: 
\begin{equation*}
\begin{split}
(\mathcal{L}_{Y_j} J_K)(p^* \xi)&=\left[Y_j, \frac{\partial}{\partial t}\right]-J_K(p^*([X_j,\xi]))=0, \\ 
(\mathcal{L}_{Y_j} J_K)\left(\frac{\partial}{\partial t}\right)&=-[Y_j, p^*\xi]-J_K\left(\left[Y_j,p^*\frac{\partial}{\partial t}\right]\right)=0,\\ 
\end{split}
\end{equation*}

\begin{equation*}
\begin{split}
(\mathcal{L}_{Y_j} J_K)(p^*X)&=[Y_j, p^*(\Phi(X))]-J_K([Y_j,p^*X])\\
&=p^*([X_j,X]-\Phi[X_j,X])=p^*((\mathcal{L}_{X_j}\Phi)(X))=0,
\end{split}
\end{equation*}
where $X$ is any vector field of the contact distribution of $W$, \emph{i.e.} $X$ is orthogonal to $\xi$.

The K\"ahler form of the K\"ahler cone is on the one hand given by: $\omega_{K}=e^{-2t}\pi^*\omega$, where $\omega$ is the fundamental $2$-form of the Vaisman structure: $\omega=g(\cdot, J\cdot)$, and on the other hand, it is exact and equals $2\mathrm{d}(e^{-2t}p^*\eta)$, where $\eta$ is the $1$-form of the Sasaki structure on $W$.
We recall that any action preserving the Sasaki structure is automatically hamiltonian with momentum map obtained by pairing with the opposite of the $1$-form $\eta$, \emph{i.e.} for any Killing vector field $X$ commuting with $\xi$, we have: $\iota_{X} d\eta=-d(\eta(X))$. The induced action  on the K\"ahler cone is also hamiltonian. More precisely, the following equation holds, for $X$ as above:
\[\iota_{(p^*X)} \omega_{K}=2\iota_{(p^*X)} \mathrm{d}(e^{-2t}p^*\eta)=-2 \mathrm{d}(e^{-2t}\eta(X)\circ p),\]
since $\mathcal{L}_X(e^{-2t}\eta)=0$. We compute for each $1\leq j\leq n$:
\begin{equation}\label{A5twham}
\begin{split}
\pi^*(\iota_{\overline{Y_j}}\omega)&=\iota_{Y_j}(\pi^*\omega)=\iota_{Y_j}(e^{2t}\omega_K)=e^{2t}\mathrm{d}(e^{-2t}\eta(X_j)\circ p)\\
&=-2(\eta(X_j)\circ p)\mathrm{d}t+d(\eta(X_j)\circ p).
\end{split}
\end{equation}
We now notice that the function $\eta(X_j)\circ p$ is invariant under the action of the group $\Gamma$. This is due, on the one hand, to the fact that $X_j$ is a fundamental vector field associated to the torus action, which by assumption commutes with $\Gamma$. On the other hand, $\eta$ equals the projection on $W$ of $J_K(\mathrm{d}t)$ and $\Gamma$ acts by biholomorphisms with respect to $J_K$ and by homotheties with respect to the cone metric, hence it automatically commutes with the radial flow $\phi_s$, \emph{cf.} Proposition~\ref{A5coneisom}.
We denote the projection of the function $\pi_*(\eta(X_j)\circ p)$ on $M$ by $\mu(X_j)=-\frac{1}{2}J\theta(\overline{Y_j})$. Hence, the right hand side of \eqref{A5twham} is projectable through $\pi$ and we obtain on $M$:
\begin{equation*}
\begin{split}
\iota_{\overline{Y_j}}\omega=-\mu(X_j)\theta+d(\mu(X_j))=\mathrm{d}^\theta(\mu(X_j)),
\end{split}
\end{equation*}
showing that the torus action $\bar\tau$ defined above is twisted Hamiltonian with momentum map defined by $\mu\colon\mathfrak{t}_n\to \mathcal{C}^\infty(M)$, $\mu(X)=-\frac{1}{2}J\theta(X_M)$, where $X_M$ is the fundamental vector field associated to $X$ through the action~$\bar\tau$. 

 \end{proof}

\section{Toric Compact Regular Vaisman manifolds} \label{A5secregvais}

We consider in this section the special case of a strongly compact regular Vaisman manifold and show that it is toric if and only if the K\"ahler quotient is toric.

Recall that a Vaisman manifold $(M,g,J,\theta)$ is called {\it regular} if the $1$-dimen\-sional distribution spanned by the Lee vector field $\theta^{\sharp}$ is regular, meaning that $\theta^{\sharp}$ gives rise to an $S^1$-action on $M$, and it is called {\it strongly regular} if both  the Lee and anti-Lee vector field give rise to an $S^1$-action.

\begin{theorem}\label{A5vaisreg}
Let $(M^{2n},g,J)$ be a strongly regular compact Vaisman manifold and denote by $\overline{M}:= M/\{\theta^\sharp,J\theta^\sharp\}$ the quotient manifold with the induced structures $\bar g$ and $\bar J$. Then $M$ is a toric Vaisman manifold if and only if $\overline{M}$ is a toric K\"ahler manifold.
\end{theorem}

\begin{proof}
Let $(M,g,J)$ be a strictly regular compact Vaisman manifold with Lee vector field $\theta^{\sharp}$ and let \mbox{$\pi\colon M\to \overline{M}$} denote the projection onto the quotient manifold. By definition, each of the vector fields $\theta^{\sharp}$ and $J\theta^{\sharp}$ gives rise to a circle action, which by Lemma~\ref{A5thJth} are both holomorphic and isometric. In this case, the metric $g$ and the complex structure $J$ descends through $\pi$ to the quotient manifold $\overline{M}$
and it is a well-known result that $(\overline{M},\bar g,\bar J)$ is a compact Hodge manifold and $\pi$ is a Riemannian submersion. 
Moreover, the curvature form of the connection $1$-form $\theta-iJ\theta$ of the principal bundle $\pi\colon M\to \overline{M}$ projects onto the K\"ahler form $\overline{\omega}$ of $\overline{M}$. For details on these results see \emph{e.g.} \cite{A5Vaisman1982} or \cite[Theorem 6.3]{A5DO1998}.

Let us first assume that $(M^{2n},g,J)$ is a compact toric Vaisman manifold, \emph{i.e.} there is an effective twisted Hamiltonian holomorphic action $\tau: T^{n}\to \mathrm{Diff}(M)$ with momentum map $\mu\colon M\to \mathbb{R}^n$. We show that there is a naturally induced effective Hamiltonian action on the compact K\"ahler quotient, $\bar\tau\colon T^{n-1}\to \mathrm{Diff}(\overline{M})$, where $T^{n-1}$ is the torus obtained from $T^n$ by quotiening out the direction corresponding to the circle action of $J\theta^\sharp$, which by Lemma~\ref{A5Jthtor} lies in the torus. Hence, by the definition of $T^{n-1}$, we have the following relation between the Lie algebras of the corresponding tori: $\mathfrak{t}_n=\mathfrak{t}_{n-1}\oplus\<J\theta^{\sharp}\>$. Since $\{\theta^{\sharp},J\theta^{\sharp}\}$ is included in the center of $\mathfrak{aut}(M)$, \emph{cf.} Lemma~\ref{A5thJth}, it follows that the action of $T^n$ on $M$ naturally descends through $\pi$ to an action on the quotient $\overline{M}= M/\{\theta^\sharp,J\theta^\sharp\}$.
By Remark~\ref{A5torvaisisom}, the action of $T^n$ on $(M,g,J)$ is both holomorphic and isometric. Thus, the action induced by $T^n$ on $(\overline{M},\bar g,\bar J)$ is also isometric and holomorphic. This can be checked as follows. By the definition of the induced action, the fundamental vector fields $X_M$ and $X_{\overline{M}}$, associated to the action of any $X\in\mathfrak{t_n}$ on  $M$, respectively on $\overline{M}$, are related by $\pi_{*}(X_M)=X_{\overline{M}}$.  Since the Lie derivative  commutes with the pull-back, we have $\mathcal{L}_{X_{\overline{M}}}\bar g=\mathcal{L}_{\pi_*(X_M)}\pi_* g=\pi_*(\mathcal{L}_{X_M}g)=0$ and similarly  we obtain $\mathcal{L}_{X_{\overline{M}}}\bar J=0$. This action of $T^n$ on $\overline{M}$ is not effective, since it is trivial in the direction of the anti-Lee vector field $J\theta^\sharp\in \mathfrak{t}_n$. However, as $\theta^{\sharp}\notin \mathfrak{t}_n$ by Lemma~\ref{A5rtw}, it follows that the restriction of this action to the above defined "complementary'' torus $T^{n-1}$ of $J\theta^{\sharp}$ in $T^n$ is effective. We denote it by $\bar\tau\colon T^{n-1}\to \mathrm{Diff}(\overline{M})$ and we only need to check that it is a Hamiltonian action on the symplectic manifold $(\overline{M}, \overline{\omega})$. For this, we consider the map $p\circ\mu\colon M\to \mathbb{R}^{n-1} $, where
$p\colon \mathbb{R}^n\to \mathbb{R}^{n-1}$ denotes the projection from $\mathfrak{t}_n$ to $\mathfrak{t}_{n-1}$, corresponding to quotiening out $\<J\theta^{\sharp}\>$, through the identification of the deals of the Lie algebras of $T^n$ and $T^{n-1}$ with $\mathbb{R}^n$, respectively $\mathbb{R}^{n-1}$. It follows that for any $X\in \mathfrak{t}_{n-1}$, the function $(p\circ\mu)(X)\colon M\to\mathbb{R}$ is invariant under the action of $\theta^{\sharp}$ and of $J\theta^{\sharp}$, hence it is projectable  through $\pi$ onto a function on $\overline{M}$, which we denote $\bar \mu^X$.  It turns out that $\bar\mu$ is the momentum map of the action $\bar\tau$. In order to show this, let $\bar x\in \overline{M}$, $\overline{Y}\in T_{\bar x} \overline{M}$, choose $x\in M$ with $\pi(x)=\bar x$ and let $Y:=(d_x\pi|_{\<\theta^{\sharp}, J\theta^{\sharp}\>^\perp})^{-1}(\overline{Y})$. We compute at the point $\bar x$, for all $X\in\mathfrak{t}_{n-1}$: $(\iota_{X_{\overline{M}}}\overline{\omega})(\overline Y)=(\pi_*\omega)(\pi_* X_M,\pi_* Y)$ and on the other hand:
$$d_{\bar x}\bar\mu^{X}(\overline{Y})=d_x\mu^{X}(Y)=\mathrm{d}^{\theta}_x\mu^{X}(Y)=(\iota_{X_M}\omega)_x(Y)=\omega_x({X_M}_x,Y),$$ thus proving that $\iota_{X_{\overline{M}}}\overline{\omega}=d\bar\mu^X$, for all $X\in\mathfrak{t}_{n-1}$. Note that by definition $\bar\mu$ inherits from $\mu$ the property of being a Lie algebra homomorphism from $\mathbb{R}^{n-1}$ endowed with the trivial Lie bracket to $\mathcal{C}^\infty(\overline{M})$ endowed with the Poisson bracket. Concluding, we have shown that the induced action $\bar\tau\colon T^{n-1}\to \mathrm{Diff}(\overline{M})$ is an effective holomorphic Hamiltonian action on $\overline{M}$, so $(\overline{M}, \bar g, \bar J)$ is a compact K\"ahler toric manifold.

Conversely, let us assume that $(M^{2n},g,J)$ is a strictly regular compact Vaisman manifold, such that the K\"ahler quotient $(\overline{M},\bar g,\bar J)$ is a toric K\"ahler manifold. We show that $(M,g,J)$ is then a toric Vaisman manifold. By assumption, $\overline{M}^{2n-2}$ is equipped with an effective Hamiltonian holomorphic action of $T^{n-1}$, which we denote by $\bar\tau\colon T^{n-1}\to \mathrm{Diff}(\overline{M})$. We also denote by $\bar\mu\colon \overline{M} \to \mathbb{R}^{n-1}$ one of the momentum maps of this action, which is unique up to an additive constant.  Let $X_1, \dots, X_{n-1}$ be a basis of the Lie algebra $\mathfrak{t}_{n-1}\cong \mathbb{R}^{n-1}$ and $f_j:=\mu^{X_j}$ the corresponding Hamiltonian functions on $\overline{M}$, \emph{i.e.} $\iota_{{\overline{X}_j}}\overline{\omega}=d f_j$, for $j\in\{1,\dots,n\}$, where for simplicity we denote the fundamental vector fields of the action on $\overline{M}$ by $\overline{X}_j=X_j^{\overline M}\in \mathfrak{X}(M)$. We now consider the horizontal distribution $\mathcal{D}$ defined for each $x\in M$ by 
$\mathcal{D}_x:=\{X\in T_x M\,|\, X\perp \theta_x^{\sharp},X\perp J\theta_x^{\sharp}\}$ and the corresponding horizontal pull-back $Y_j$ of each vector field $X_j$ through $\pi$,
\emph{i.e.} for each $x\in M$, $(Y_j)_x:=((d_x\pi)|_{\mathcal{D}_x})^{-1}((\overline{X_j})_{\pi(x)})$, for $j\in\{1,\dots,n\}$. Each vector field $Y_j$ is complete and thus gives rise to an action of $\mathbb{R}$ on $M$. In the sequel, we modify them in the direction of $J\theta^{\sharp}$ in order to obtain the lifted torus action on $M$. 

\noindent Set $Z_j:=Y_j-(f_j\circ \pi) J\theta^{\sharp}$, for $1\leq j\leq n-1$. 

\noindent {\bf Step 1.} We first show that each of these vector fields commutes with $\theta^{\sharp}$ and with $J\theta^{\sharp}$ and preserves the distribution $\mathcal{D}$, since it leaves invariant both $1$-forms $\theta$ and $J\theta$. Namely, we compute for $1\leq j\leq n-1$:
\begin{equation}\label{A5commzth}
 [Z_j,\theta^{\sharp}]=\left[Y_j-f_j\circ \pi \cdot J\theta^{\sharp},\theta^{\sharp}\right]=0,\quad
 [Z_j,J\theta^{\sharp}]=\left[Y_j-f_j\circ \pi \cdot J\theta^{\sharp},J\theta^{\sharp}\right]=0,
 \end{equation}
where we use that all functions $f_j\circ \pi$ and all horizontal lifts $Y_j$ are by definition constant along the flows of $\theta^{\sharp}$ and of $J\theta^{\sharp}$. We further obtain by applying the Cartan formula:
\[\mathcal{L}_{Z_j}\theta=\iota_{Z_j} d\theta+d(\iota_{Z_j}\theta)=0,\] 
since $\theta$ is closed and $Z_j$ is by definition in the kernel of $\theta$. 
Since the curvature of the connection $1$-form of the principal bundle 
$\pi\colon M\to\overline{M}$ projects onto the K\"ahler form $\overline{\omega}$, we obtain:
\[\mathcal{L}_{Z_j}(J\theta)=\iota_{Z_j} d(J\theta)+d(\iota_{Z_j}(J\theta))=\iota_{Z_j}\pi^*\overline{\omega}-d(f_j\circ\pi)=\pi^*(\iota_{\overline{X}_j}\overline{\omega})-\pi^*(df_j)=0.\] 
Indeed, in order to obtain this last equality it was necessary to perturb the horizontal lifts $Y_j$ in the direction of $J\theta^{\sharp}$. We claim that  the set  $\{Z_1, \cdots, Z_{n-1}\}$ gives rise to an effective action of $\mathbb{R}^{n-1}$ on $M$. For this, it is sufficient to check that all their Lie brackets vanish, so they are commuting vector fields. The property of being an effective action is a consequence of the effectiveness of the action of $T^{n-1}$ on $\overline{M}$. 
For each $j,k\in\{1,\dots, n-1\}$ we compute:
\begin{equation}\label{A5comm}
\begin{split}
[Z_j,Z_k]&=\left[Y_j-f_j\circ \pi\cdot J\theta^{\sharp}, Y_k-f_k\circ \pi\cdot J\theta^{\sharp}\right]\\
&=[Y_j,Y_k]-Y_j\left(f_k\circ \pi\right)J\theta^{\sharp}+Y_k\left(f_j\circ \pi\right)J\theta^{\sharp}\\
&=[\overline{X}_j,\overline{X}_k]^*-\left(\overline{\omega}(\overline{X}_j, \overline{X}_k)-d f_j(\overline{X}_k)+d f_k(\overline{X}_j)\right)\circ\pi  \cdot J\theta^{\sharp}\\
&=[\overline{X}_j,\overline{X}_k]^*-\overline{\omega}(\overline{X}_k, \overline{X}_j)\circ\pi  \cdot J\theta^{\sharp}=0,
\end{split}
\end{equation}
where by $^*$ is denoted the horizontal lift to the distribution $\mathcal{D}$ of a vector field on $\overline{M}$.  For the above equalities we used again that $J\theta^{\sharp}(f_j\circ \pi)=0$ and $[Y_j,J\theta^{\sharp}]=0$, for all $1\leq j\leq n-1$, as well as the fact that the curvature of the connection $1$-form of the principal bundle 
equals $\pi^*\overline{\omega}$. Note that the last equality in \eqref{A5comm} follows from the commutation of the  vector fields $\overline{X}_j$ induced by the torus action and from the property of the momentum map $\bar\mu\colon \mathbb{R}^{n-1}\to \mathcal{C}^\infty(\overline{M})$ of being a Lie algebra homomorphism, where $\mathcal{C}^\infty(\overline{M})$ is endowed with the Poisson bracket. Hence, for any $j,k$, the following relations hold: $0=\bar\mu([X_j,X_k])=\{\bar\mu^{X_j}, \bar\mu^{X_k}\}=\overline{\omega}(\overline{X}_j,\overline{X}_k)$.

\noindent {\bf Step 2.} We now check that the above defined action of $\mathbb{R}^{n-1}$ on $M$ is holomorphic, isometric and twisted Hamiltonian.  For each $1\leq j\leq n-1$, $Z_j$ is a holomorphic vector field by the following computation, which uses the commutation relations from \eqref{A5commzth}:
\[(\mathcal{L}_{Z_j} J)(\theta^{\sharp})=[Z_j,J\theta^{\sharp} ]-J[Z_j,\theta^{\sharp} ]=0, \;  (\mathcal{L}_{Z_j} J)(J\theta^{\sharp})=-[Z_j,\theta^{\sharp} ]-J[Z_j,J\theta^{\sharp} ]=0,\]
and since for each projectable horizontal vector field $Y\in\Gamma(\mathcal{D})$ with $\pi_* Y=\overline{Y}$ we have:
\begin{equation*}
\begin{split}
(\mathcal{L}_{Z_j} J)(Y)=&[Y_j-(f_j\circ \pi)J\theta^{\sharp} ,JY]-J[Y_j-(f_j\circ \pi)J\theta^{\sharp},Y]\\
=&[Y_j,JY]-(f_j\circ \pi) [J\theta^{\sharp} ,JY]+JY(f_j\circ \pi)J\theta^{\sharp} \\
&-J[Y_j,Y]+(f_j\circ \pi)J[J\theta^{\sharp},Y]+Y(f_j\circ \pi)\theta^{\sharp}\\
=&[\overline{X}_j,\overline{J}\, \overline{Y}]^*-\overline{\omega}(\overline{X}_j,\overline{J}\, –\overline{Y})\circ\pi \cdot J\theta^{\sharp}+d f_j(\overline{J}\, \overline{Y})\circ\pi \cdot J\theta^{\sharp}\\
&-J([\overline{X}_j, \overline{Y}]^*-\overline{\omega}(\overline{X}_j,–\overline{Y})\circ\pi \cdot J\theta^{\sharp})+d f_j( \overline{Y})\circ\pi \cdot \theta^{\sharp}\\
=&([\overline{X}_j,\overline{J}\, \overline{Y}]-\overline{J}[\overline{X}_j, \overline{Y}])^*=0,
\end{split}
\end{equation*}
where we use that by definition $\overline{J}$ is the projection of $J$, so $JY=(\overline{J}\, –\overline{Y})^*$, as well as the equalities $[J\theta^{\sharp}, JY]=J[J\theta^{\sharp},Y]$, following from the fact that $J\theta^{\sharp}$ is a holomorphic vector field, \emph{cf.} Lemma \ref{A5thJth}. A similar computation yields that $\mathcal{L}_{Z_j} g=0$, for $1\leq j\leq n-1$, showing that each $Z_j$ is a Killing vector field on $M$ with respect to $g$. 
We further compute for each $j\in\{1,\dots, n-1\}$:
\begin{equation*}
\begin{split}
\iota_{Z_j}\omega&=\iota_{Y_j}\omega-(f_j\circ\pi)\iota_{J\theta^{\sharp}}\omega=\pi^*(\iota_{\overline{X}_j}\overline{\omega})-(f_j\circ\pi)\theta\\
&=d (f_j\circ\pi)-(f_j\circ\pi)\theta=\mathrm{d}^{\theta}(f_j\circ\pi),
\end{split}
\end{equation*}
by taking into account that $\omega=\pi^*\overline{\omega}-\theta\wedge J\theta$. This proves that each vector field $Z_j$ is twisted Hamiltonian with respect to $\omega$ with Hamiltonian function $f_j\circ\pi$. Altogether, we have shown that the action of $\mathbb{R}^{n-1}$ defined by the vector fields $Z_1, \dots, Z_j$ respects the whole Vaisman structure $(g,J,\omega)$ and is twisted Hamiltonian with momentum map $\mu
\colon \mathbb{R}^{n-1}\to\mathcal{C}^\infty(M)$, defined by $\mu^{X_j}:=\bar\mu^{X_j}\circ \pi$, for $1\leq j\leq n-1$. We also remark that $\mu$ is a Lie algebra homomorphism for $\mathcal{C}^\infty(M)$ endowed with the Poisson bracket defined by \eqref{A5Pb}, as $\bar\mu$ is a Lie algebra homomorphism and the Poisson brackets are compatible through $\pi$. More precisely, we have for all $j,k\in\{1,\dots, n-1\}$: 
\begin{equation*}
\begin{split}
\{\mu^{X_j}, \mu^{X_k}\}\overset{\eqref{A5Pb}}{=}&\omega(Z_j,Z_k)=\omega(Y_j-(f_j\circ\pi) J\theta^{\sharp} ,Y_k-(f_k\circ\pi) J\theta^{\sharp})=\\
=&\omega(Y_j, Y_k)-(f_j\circ\pi)\omega(J\theta^{\sharp}, Y_k)+(f_k\circ\pi)\omega(J\theta^{\sharp}, Y_j)\\
=&\overline\omega(\overline{X}_j, \overline{X}_k)\circ \pi-(f_j\circ\pi) \theta(Y_k)+(f_k\circ\pi) \, \theta(Y_j)\\
=&\{\bar\mu^{X_j}, \bar\mu^{X_k}\}\circ\pi
=\bar\mu^{[X_j,X_k]}\circ \pi=\mu^{[X_j,X_k]}.
\end{split}
\end{equation*}

\noindent {\bf Step 3.} Next we show that the freedom in choosing the momentum map $\bar\mu$ of $\overline{M}$ up to an additive constant in $\mathbb{R}^{n-1}$ allows us to make the above defined action of $\mathbb{R}^{n-1}$ on $M$ into an action of $T^{n-1}=S^1\times \cdots \times S^1$ on $M$, thus lifting the action of $T^{n-1}$ on $\overline{M}$. It is enough to establish this result for one direction and then apply it independently for each of the directions $X_j$, $j\in\{1,\dots, n-1\}$. The argument  can be found in \cite{A5Gauduchon-Extremal-Calabi}, Section 7.5., in the more general setting of lifting $S^1$-actions to Hermitian complex line bundles equipped with a $\mathbb{C}$-linear connection, whose curvature form is preserved by the circle action and in which case the property of being Hamiltonian is with respect to the closed curvature $2$-form, which is not assumed to be symplectic. For the convenience of the reader we sketch here this argument using the above notation.
Let $X$ be a vector field on $\overline M$ which generates an $S^1$-action, whose orbits are assumed to have period $1$. One considers the flow $\Phi_s$ of the horizontal lift of $X$ to 
$M/\{\theta^\sharp\}$, which then satisfies: $\Phi_1(x)=\zeta(\bar x) \cdot x$, for all $x\in M/\{\theta^\sharp\}$, where $\bar x$ is the projection of $x$ on $\overline M$ and $\zeta\colon\overline M\to S^1$. The main step is to show that $\zeta$ is constant. This follows from the following formula $\Phi_1 \cdot V=V-i\frac{\mathrm{d}\zeta(\overline V)}{\zeta}J\theta^\sharp$, for all $\overline V \in\mathfrak{X}(\overline M)$ and $V$ its horizontal lift to $M/\{\theta^\sharp\}$,  and from the fact that $\Phi_s$ preserves the horizontal distribution for all values of $s$. An $S^1$-action on $M/\{\theta^\sharp\}$ lifts to an $S^1$-action on $M$, since $\theta$ is closed, so its kernel defines an integrable distribution on~$M$. Applied to our case, this result yields that for each $1\leq j\leq n-1$, the Hamiltonian function $f_j\in\mathcal{C}^{\infty}(M)$ may be chosen (by adding an appropriate real constant, determined by $\zeta$, which is well-defined up to an additive integer), such that the vector field $Z_j=Y_j-(f_j\circ\pi) \cdot J\theta^{\sharp}$ is the generator of an $S^1$-action on $M$, where $Y_j$ denotes the horizontal lift of $\overline{X}_j$. 

Since $J\theta^{\sharp}$ is a Killing, holomorphic, twisted Hamiltonian vector field lying in the center of $\mathfrak{aut}(M)$, \emph{cf.} Lemma \ref{A5thJth} and Lemma~\ref{A5twistedham}, it follows that the circle action induced by $J\theta^{\sharp}$ commutes with the circle actions induced by the above defined action of $T^{n-1}$ on $M$, thus giving rise together to an effective holomorphic twisted Hamiltonian action of $T^n$ on $(M^{2n},g,J)$, showing that this is a toric Vaisman manifold.\end{proof}

Applying this result to the Vaisman manifold $S^1\times S^{2n-1}$  as described in Example \ref{A5example}, which is strongly regular, as $\theta^\sharp$ is the tangent vector to the $S^1$ factor and 
$J\theta^\sharp$ is tangent to the Hopf action on $S^{2n-1}$. This action commutes 
with the torus action defined in  Example \ref{A5example}. It follows that this $T^n$-action descends to the quotient 
$(S^1\times S^{2n-1})/\{\theta^\sharp, J\theta^\sharp\}\simeq \mathbb C P^n$. 

\begin{remark}\label{A5twostep}
If $(M^{2n},g,J)$ is a regular compact Vaisman manifold, then the proof of Theorem~\ref{A5vaisreg} also shows that the following equivalence holds: $M$ is a toric Vaisman manifold if and only if $M/\{\theta^{\sharp}\}$ is a toric Sasaki manifold. Again, this statement is only about the toric structure, since it is well-known that the quotient manifold is Sasaki (see \emph{e.g.} \cite{A5Vaisman1982} or \cite{A5DO1998}). 

\end{remark}

We note that the proofs of the theorems of the last two sections are  both  constructive, showing that the torus actions on the  Sasaki (respectively K\"ahler) manifold and on the Vaisman manifold naturally induce one another.

\bibliographystyle{amsplain}

\begin{thebibliography}{10}

\bibitem{A5Abreu1998}
M.~Abreu, \emph{K\"ahler geometry of toric varieties and extremal metrics},
  Internat. J. Math. \textbf{9} (1998), no.~6, 641--651.

\bibitem{A5Abreu2001}
\bysame, \emph{K\"ahler metrics on toric orbifolds}, J. Differential Geom.
  \textbf{58} (2001), no.~1, 151--187.

\bibitem{A5Abreu2010}
\bysame, \emph{K\"ahler-{S}asaki geometry of toric symplectic cones in
  action-angle coordinates}, Port. Math. \textbf{67} (2010), no.~2, 121--153.

\bibitem{A5ACG2006}
V.~Apostolov, D.~M.~J. Calderbank, and P.~Gauduchon, \emph{Hamiltonian 2-forms
  in {K}\"ahler geometry. {I}. {G}eneral theory}, J. Differential Geom.
  \textbf{73} (2006), no.~3, 359--412.

\bibitem{A5BM2015}
F.~Belgun and A.~Moroianu, \emph{On the irreducibility of locally metric
  connections}, J. reine angew. Math. (2016).

\bibitem{A5BD2000}
R.~Bielawski and A.S. Dancer, \emph{The geometry and topology of toric
  hyperk\"ahler manifolds}, Comm. Anal. Geom. \textbf{8} (2000), no.~4,
  727--760.

\bibitem{A5BG2008}
C.~Boyer and K.~Galicki, \emph{Sasakian geometry}, Oxford Mathematical
  Monographs, no. xii+613, Oxford University Press, Oxford, 2008.

\bibitem{A5CDG2003}
D.M.~J. Calderbank, L.~David, and P.~Gauduchon, \emph{The {G}uillemin formula
  and {K}\"ahler metrics on toric symplectic manifolds}, J. Symplectic Geom.
  \textbf{1} (2003), no.~4, 767--784.

\bibitem{A5Dasilva2001}
A.~Cannas da~Silva, \emph{Lectures on symplectic geometry}, Lecture Notes in
  Mathematics, vol. 1764, Springer-Verlag, Berlin, 2001.

\bibitem{A5Delzant1988}
T.~Delzant, \emph{Hamiltoniens p\'eriodiques et images convexes de
  l'application moment}, Bull. Soc. Math. France \textbf{116} (1988), no.~3,
  315--339.

\bibitem{A5DO1998}
S.~Dragomir and L.~Ornea, \emph{Locally conformal {K}\"ahler geometry},
  Progress in Mathematics, vol. 155, Birkh\"auser Boston, Inc., Boston, MA,
  1998.

\bibitem{A5FOW2009}
A.~Futaki, H.~Ono, and G.~Wang, \emph{Transverse {K}\"ahler geometry of
  {S}asaki manifolds and toric {S}asaki-{E}instein manifolds}, J. Differential
  Geom. \textbf{83} (2009), no.~3, 585--635.

\bibitem{A5Gauduchon-Extremal-Calabi}
P.~Gauduchon, \emph{Calabi's extremal k\"ahler metrics: An elementary
  introduction}, Lecture Notes, unpublished.
  
  \bibitem{A5Gauduchon-Moroianu2016}
P.~Gauduchon and A.~Moroianu, \emph{Weyl-Einstein structures on $K$-contact manifolds}, arXiv: 1601.00892.
  

\bibitem{A5GOP2005}
R.~Gini, L.~Ornea, and M.~Parton, \emph{Locally conformal {K}\"ahler
  reduction}, J. Reine Angew. Math. \textbf{581} (2005), 1--21.

\bibitem{A5GOPP2006}
R.~Gini, L.~Ornea, M.~Parton, and P.~Piccinni, \emph{Reduction of {V}aisman
  structures in complex and quaternionic geometry}, J. Geom. Phys. \textbf{56}
  (2006), no.~12, 2501--2522.

\bibitem{A5Guillemin1994}
V.~Guillemin, \emph{K\"ahler structures on toric varieties}, J. Differential
  Geom. \textbf{40} (1994), no.~2, 285--309.

\bibitem{A5HR2001}
S.~Haller and T.~Rybicki, \emph{Reduction for locally conformal symplectic
  manifolds}, J. Geom. Phys. \textbf{37} (2001), no.~3, 262--271.

\bibitem{A5KL2015}
Y.~Karshon and E.~Lerman, \emph{Non-compact symplectic toric manifolds}, SIGMA
  Symmetry Integrability Geom. Methods Appl. \textbf{11} (2015), Paper 055, 37.

\bibitem{A5L2003}
E.~Lerman, \emph{Contact toric manifolds}, J. Symplectic Geom. \textbf{1}
  (2003), 785--828.

\bibitem{A5LT1997}
E.~Lerman and S.~Tolman, \emph{Hamiltonian torus actions on symplectic
  orbifolds and toric varieties}, Trans. Amer. Math. Soc. \textbf{349} (1997),
  no.~10, 4201--4230.

\bibitem{A5MSY2006}
D.~Martelli, J.~Sparks, and S.-T. Yau, \emph{The geometric dual of
  {$a$}-maximisation for toric {S}asaki-{E}instein manifolds}, Comm. Math.
  Phys. \textbf{268} (2006), no.~1, 39--65.

\bibitem{A5MO2009}
A.~Moroianu and L.~Ornea, \emph{Transformations of locally conformally
  {K}\"ahler manifolds}, Manuscripta Math. \textbf{130} (2009), no.~1, 93--100.

\bibitem{A5Otiman}
A.~Otiman, \emph{Locally conformally symplectic bundles}, preprint (2015),
  arXiv:1510.02770.

\bibitem{A5Vaisman1979}
I.~Vaisman, \emph{Locally conformal {K}\"ahler manifolds with parallel {L}ee
  form}, Rend. Mat. (6) \textbf{12} (1979), no.~2, 263--284.

\bibitem{A5Vaisman1982}
\bysame, \emph{Generalized {H}opf manifolds}, Geom. Dedicata \textbf{13}
  (1982), no.~3, 231--255.

\bibitem{A5Vaisman1985}
\bysame, \emph{Locally conformal symplectic manifolds}, Internat. J. Math.
  Math. Sci. \textbf{8} (1985), no.~3, 521--536.

\bibitem{A5vC2009}
C.~van Coevering, \emph{Some examples of toric {S}asaki-{E}instein manifolds},
  Riemannian topology and geometric structures on manifolds, Progr. Math., vol.
  271, Birkh\"auser Boston, Boston, MA, 2009, pp.~185--232.

\end{thebibliography}

\end{document}